\documentclass[12pt]{iopart}

\usepackage{amsthm}
\usepackage{tikz-cd}
\usetikzlibrary{arrows}
\usepackage{enumitem}
\usepackage{graphicx}

\newcommand{\mc}{\mathcal}
\newcommand{\ts}{\textsc}
\def\Perm{\mbox{Perm}}

\newtheorem{thm}{Theorem}[section]
\newtheorem{lemma}[thm]{Lemma}
\newtheorem{cor}[thm]{Corollary}
\newtheorem{prop}[thm]{Proposition}

\newtheorem{defn}[thm]{Definition}
\newtheorem{remark}[thm]{Remark}
\newtheorem{example}[thm]{Example}

\begin{document}

\title{Quantum monadic algebras}

\author{J.~Harding \footnote{The author gratefully acknowledges support of US Army grant W911NF-21-1-0247. }}

\address{New Mexico State University, Las Cruces, NM 88003, USA}
\ead{jharding@nmsu.edu}
\vspace{10pt}
\begin{indented}
\item[]February 2022
\end{indented}


\date{}

\begin{abstract}
We introduce quantum monadic and quantum cylindric algebras. These are adaptations to the quantum setting of the monadic algebras of Halmos, and cylindric algebras of Henkin, Monk and Tarski, that are used in algebraic treatments of classical and intuitionistic predicate logic. Primary examples in the quantum setting come from von Neumann algebras and subfactors. Here we develop the basic properties of these quantum monadic and cylindric algebras and relate them to quantum predicate logic. 
\end{abstract}

\vspace{2pc}
\noindent{\it Keywords}: cylindric algebra, monadic algebra, quantum logic, subfactor, von Neumann algebra

%
%

\section{Introduction}

Beginning with the work of Boole  
in the mid 1800's, algebraic techniques have played a significant role in logic. The Boolean algebras that grew from Boole's work serve as algebraic models of classical propositional calculus. However, the significance of Boolean algebras has grown past this original motivation. Boolean algebras are a recognized feature in such diverse areas as digital electronics, set theory, probability and measure theory, commutative algebra, and operator theory. In addition to the role played by Boolean algebras within each subject is the important role they play in moving ideas between these subjects. 

The algebraic techniques used to model classical propositional logic were extended to algebraic treatments of intuitionistic propositional logic by Heyting and to modal propositional logic by MacColl, Lewis, and others. 
In the first case, the relevant algebras are Heyting algebras, and in the second Boolean algebras with additional operations. 

The algebraic treatment of classical propositional logic was extended in another direction to give an algebraic view of classical first order logic, also known as predicate calculus. Roughly, predicate calculus extends propositional logic by allowing quantification over individuals using $\forall x$ and $\exists x$. 

There were two closely related streams in this direction, one originated by Halmos \cite{Halmos62}, and the other by Tarski \cite{HMT85a,HMT85b}. Halmos first treated the 1-variable fragment of predicate calculus via what he termed a {\em monadic algebra}, that is, a Boolean algebra with an additional unary operation $\exists$. Halmos' extension to an arbitrary number of variables came under the name of a {\em polyadic algebra}, a Boolean algebra with a family of unary operations satisfying certain properties. Tarski's {\em cylindric algebras} were closely related to Halmos' polyadic algebras and again consisted of a Boolean algebra with a family of additional operations. 

It their 1937 paper \cite{BvN37}, Birkhoff and von Neumann proposed that the projection lattice $P(\mc{H})$ of a Hilbert space $\mc{H}$ be used as an algebraic semantics for a type of propositional logic arising in quantum mechanics just as Boolean algebras are used as an algebraic semantics for classical propositional logic. They termed this the ``logic'' of quantum mechanics. This process of replacing a Boolean algebra with a projection lattice, or more generally an orthomodular lattice \cite{Kalmbach83}, can be viewed as the first of many ``quantizations'' of classical mathematical notions. Others include such notions as quantum probability, non-commutative geometry 
and quantum groups. 

Formal propositional logical systems can be built from projection lattices $P(\mc{H})$, or general orthomodular lattices, rather than Boolean algebras. This gives an obvious meaning to the term ``quantum logic''. There is a more subtle meaning of the term ``quantum logic'', and likely the more prevalent one. This refers to viewing the algebraic structure $P(\mc{H})$ as a vehicle to to carry ideas between areas, much as one looks at Boolean algebras in this way. 

As cornerstones of this general view of quantum logic, 
the spectral theorem, Gleason's theorem, Wigner's theorem, and Stone's theorem describe self-adjoint operators, states, unitary and anti-unitary operators, representations, and dynamics of a quantum system in terms of certain natural mappings involving the projection lattice $P(\mc{H})$ (see \cite{KR97a,Dvurecenskij93,Uhlhorn63,Harding17} for further details). 
In addition, the structure of $P(\mc{H})$ has strong ties to projective geometry that form the basis of the ``geometric'' approach to quantum mechanics of Varadarajan \cite{Varadarajan68}. 

In this note, we introduce quantum analogs of the monadic and cylindric algebras that are used to study classical predicate calculus. These {\em quantum monadic algebras} and {\em quantum cylindric algebras} are projection lattices $P(\mc{H})$, or general orthomodular lattices, equipped with additional operations satisfying conditions similar to their classical counterparts. We develop the basic theory of such algebras, including a treatment of their dual spaces in the spirit of Kripke frames \cite{brv01}, and detail the use of these algebras for an algebraic treatment of quantum predicate calculus. This treatment of quantum predicate calculus involves a detailed view of the projection lattice $P(\mc{H}^{\otimes n})$ of a tensor power of a Hilbert space $\mc{H}$. 

Further examples of quantum monadic algebras arise from von Neumann algebras. Any concrete representation of a von Neumann algebra $\mc{M}$ in the bounded operators $B(\mc{H})$ of a Hilbert space $\mc{H}$ provides a quantum monadic algebra structure on the projection lattice $P(\mc{H})$, and any subfactor $\mc{M}\leq\mc{N}$ provides a quantum monadic algebra structure on the projection lattice $P(\mc{N})$. Additionally, it appears there are ties between quantum cylindric algebras and the commuting squares \cite{Popa83,WW95} of projections that provide a quantum analog of independent $\sigma$-fields. It is our hope that quantum monadic and cylindric algebras can eventually serve as a vehicle to import logical perspective into the study of these operator algebras. 

This paper is arranged in the following way. Sections 2--4 provide background on classical monadic and cylindric algebras, von Neumann algebras, and quantum logic. The fifth section introduces quantum monadic and quantum cylindric algebras and develops their basic properties. The sixth section gives a treatment of quantum predicate calculus in terms of quantum cylindric algebras. The seventh section deals with an analog of Kripke frame representations of these algebras. The final section concludes with some comments and suggestions for further directions. 

\section{Classical monadic and cylindric algebras} 

For a detailed account of monadic and polyadic algebras see \cite{Halmos62} and for an account of cylindric algebras see \cite{HMT85a,HMT85b}. Throughout we use $\wedge, \vee,\perp,0,1$ for the meet, join, complementation, and bounds of a Boolean algebra (abbrev. \ts{ba}). 

\begin{defn}\label{defn:monadic}
A {\em quantifier} $\exists$ on a \ts{ba} $B$ is a unary operation that satisfies
\begin{itemize}[leftmargin=.4in]
\item[{\em \small (Q$_1$)}] $\exists 0 = 0$,
\item[{\em \small (Q$_2$)}] $p\leq \exists p$,
\item[{\em \small (Q$_3$)}] $\exists(p\vee q) = \exists p \vee \exists q$,
\item[{\em \small (Q$_4$)}] $\exists\exists p = \exists p$,
\item[{\em \small (Q$_5$)}] $\exists (\exists p)^\perp = (\exists p)^\perp$.
\end{itemize}
A {\em monadic algebra} $(B,\exists)$ is a \ts{ba} $B$ with a quantifier $\exists$.
\end{defn}

\begin{lemma} {\em \cite{Halmos62}} \label{lemma:a}
The conditions {\em \small (Q$_1$) -- (Q$_5$)} are equivalent to {\em\small (Q$_1$), (Q$_2$), (Q$_6$)} where 
\begin{itemize}[leftmargin=.4in]
\item[{\em \small (Q$_6$)}] $\exists(p\wedge\exists q) = \exists p \wedge \exists q$. 
\end{itemize}
\end{lemma}

One can similarly give a definition of a universal quantifier $\forall$. This definition is dual to that given above, replacing $0$ with 1, $\leq$ with $\geq$, and join with meet. The relation between universal and existential quantifiers is the expected, if $\exists$ is an existential quantifier, then $\forall p = (\exists p^\perp)^\perp$ is a universal quantifier and conversely. 

\begin{example} {\em
Halmos' motivating example is that of a {\em functional monadic algebra}. Suppose individuals for the single variable we wish to quantify over range in a set $X$. Then a proposition depending on this variable has a truth value for each member of $X$, and we assume these truth values belong to a complete \ts{ba} $B$. So propositions give functions $f:X\to B$. The quantifier $\exists f$ applied to the {\em propositional function} $f$ is defined to take constant value $\bigvee\{f(x):x\in X\}$ and the universal quantifier $\forall f$ applied to $f$ takes constant value $\bigwedge\{f(x):x\in X\}$. A key step in Halmos' development of the theory of monadic algebras is to show that each monadic algebra is a subalgebra of a functional one. 
}
\end{example}

Halmo extends monadic algebras to allow quantification over more than one variable with his {\em polyadic algebras} \cite{Halmos62}. This notion is closely related to the cylindric algebras initiated by Tarski, and developed by Henkin, Monk and Tarski \cite{HMT85a,HMT85b}. Roughly, cylindric algebras amount to polyadic algebras with equality \cite{Galler57}. We prefer to present things from the more common perspective of cylindric algebras. Their presentation here is modified somewhat to take advantage of our current development of quantifiers. 

\begin{defn}
A {\em cylindric algebra} $(B,c_i, d_{i,j})$ consists of a \ts{ba} $B$ with a family $c_i$ of unary operations and $d_{i,j}$ of constants for $i,j\in I$ such that for each $i,j,k\in I$ 
\begin{itemize}[leftmargin=.4in]
\item[{\em \small (C$_1$)}] $c_i$ is a quantifier,
\item[{\em \small (C$_2$)}] $c_ic_jx=c_jc_ix$,
\item[{\em \small (C$_3$)}] $d_{i,j}=d_{j,i}$ and $d_{i,i}=1$,
\item[{\em \small (C$_4$)}] if $j\neq i,k$ then $d_{i,k}=c_j(d_{i,j}\wedge d_{j,k})$,
\item[{\em \small (C$_5$)}] if $i\neq j$ then $c_i(d_{i,j}\wedge x)\wedge c_i(d_{i,j}\wedge x^\perp) = 0$.
\end{itemize}
The operations $c_i$ are called cylindrifications, and the $d_{i,j}$ are diagonals. 
\end{defn}

Usually the indexing set $I$ is assumed to be an ordinal and called the {\em dimension} of the cylindric algebra. A {\em diagonal-free} cylindric algebra is simply a \ts{ba} with a family of pairwise commuting quantifiers.  

\begin{example}\label{ex:cylindric set algebra} {\em 
The motivating example of an $I$-dimensional cylindric algebra comes from considering the the \ts{ba} $B$ of all subsets of $X^I$ for some set $X$. Elements $f\in X^I$ are $I$-indexed families in $X$. For a subset $A\subseteq B$, the cylindrification $c_i A$ is the set of all such indexed families $g$ that agree with some $f\in A$ except possibly in the coordinate $i$. The diagonal $d_{i,j}=\{f\mid f(i)=f(j)\}$. 
}
\end{example}

\begin{example} {\em 
One obtains an $\omega$-dimensional cylindric algebra, where $\omega$ is the natural numbers, by considering a first-order structure based on a set $X$. For each formula $\phi$ consider the subset of $X^\omega$ of all $\omega$-indexed families $f$ that satisfy $\phi$ when $f(n)$ is used for all free occurrences of the variable $x_n$ in $\phi$. Cylindrification and diagonals are as in the previous example. 
}
\end{example}

\begin{example} \label{ex3} {\em 
An $\omega$-dimensional cylindric algebra is obtained by taking a set $\Sigma$ of sentences in a first-order language and considering equivalence classes of formulas under the relation $\Sigma\vdash \phi \leftrightarrow \psi$ and usual ordering given by implication. This is a \ts{ba} with cylindrification $c_n$ of the equivalence class of $\phi$ being the equivalence class $\exists_n \phi$ and the diagonal $d_{m,n}$ being the equivalence class of the formula $x_m=x_n$. 
}
\end{example}

For further details of these examples, see \cite[p.~6-10]{HMT85a} and \cite[Ch.~4]{HMT85b}. 

\begin{defn} \label{defn:substitution}
For an $I$-dimensional cylindric algebra, define for each $i,j\in I$ the {\em substitution operator} $S_j^{\, i} x = c_i(d_{i,j}\wedge x)$ for $i\neq j$ and $S_{i}^{\, i} x = x$. 
\end{defn}

When interpreted in the cylindric algebra of Example~\ref{ex3}, the substitution operator $S_j^{\, i}$ applied to the equivalence class of a formula $\phi$ is the equivalence class of the formula obtained from $\phi$ by substituting the variable $x_j$ in place of all free occurrences of $x_i$ in the formula $\phi$. Substitution operators are Boolean endomorphisms of the cylindric algebra, as one would expect of substitution. They play an important role in the theory, and are used to develop the connection between relation algebras and cylindric algebras \cite[Ch.~4]{HMT85b}. In fact, the notion of cylindric algebras can be developed by taking these substitution operators as primitive \cite{Cirulis94}. 

\section{von Neumann algebras}

Our purpose is treat several details of von Neumann algebras that feature in later sections. We assume a basic familiarity with the topic as is found in \cite{KR97a,KR97b}. In particular, a von Neumann algebra $\mc{M}$ is a unital $*$-subalgebra of the bounded operators $B(\mc{H})$ of a Hilbert space $\mc{H}$ that is closed in the weak operator topology. Equivalently, it is a subalgebra of $B(\mc{H})$ that is equal to its double commutant $\mc{M}^{cc}$ where the commutant $S^c$ of $S\subseteq B(\mc{H})$ is the collection of all operators commuting with each operator in $S$. 

\begin{defn}
A self-adjoint projection of a von Neumann algebra $\mc{M}$ is an element $p\in\mc{M}$ with $p=p^*=p^2$. 
\end{defn}

We often refer to these simply as projections since we do not consider others, and use $P(\mc{M})$ for the projections of $\mc{M}$. If $\mc{M}$ is a subalgebra of $B(\mc{H})$, then projections of $\mc{M}$ are projection operators of $\mc{H}$, so correspond to closed subspaces of $\mc{H}$. An essential point of difference between von Neumann algebras and general C$^*$-algebras is that von Neumann algebras have a rich supply of projections that largely determine their behavior. 

\begin{defn}
An ortholattice (abbrev. \ts{ol}) is a bounded lattice $L$ with a unary operation $\perp$ that is order-inverting, period two, and satisfies $x\wedge x^\perp = 0$ and $x\vee x^\perp = 1$. An orthomodular lattice (abbrev. \ts{oml}) is an \ts{ol} that satisfies $x\leq y \Rightarrow x\vee (x^\perp\wedge y) = y$. 
\end{defn}

For details of \ts{oml}s, see \cite{Kalmbach83}. There is a close tie between von Neumann algebras and \ts{oml}s. To describe this, we recall that $C$ is a {\em complete sublattice} of a complete lattice $L$ if $C$ is closed under all (even infinite) joins and meets of $L$. 

\begin{thm}
If $\mc{M}$ is a von Neumann subalgebra of the bounded operators $B(\mc{H})$ of a Hilbert space $\mc{H}$, then $P(\mc{M})$ is a complete sub-ortholattice of the complete \ts{oml} $P(\mc{H})$ of projections of $\mc{H}$. Further, $\mc{M}$ is generated as a von Neumann algebra by its projections. 
\end{thm}

This theorem is found in \cite{KR97a,KR97b,Hamhalter03}. In fact, much more is true, although we won't need it. Modulo certain small pathologies, Dye's theorem shows that the Jordan structure (structure under the symmetrized product $x\circ y = \frac{1}{2}(xy+yx)$) is determined by the projections $P(\mc{M})$. See \cite{Hamhalter03} for details. 

\begin{defn} 
The tensor product $\mc{H}\otimes\mc{K}$ of Hilbert spaces $\mc{H}$ and $\mc{K}$ is the metric space completion of their algebraic tensor product. 
\end{defn}

If $(a_i)_I$ is an orthonormal basis (abbrev. \ts{onb}) of $\mc{H}$ and $(b_j)_J$ is an \ts{onb} of $\mc{K}$, it is well-known that $(a_i\otimes b_j)_{I\times J}$ is an \ts{onb} of $\mc{H}\otimes\mc{K}$. This is the case whether the spaces are separable or not. 
Also, if $A$ is a bounded operator of $\mc{H}$ and $B$ is a bounded operator of $\mc{K}$, then there is a bounded operator $A\otimes B$ on the tensor product extending the algebraic tensor product of $A$ and $B$. 

\begin{defn}
If $\mc{M}$ is a von Neumann subalgebra of $B(\mc{H})$ and $\mc{N}$ is a von Neumann subalgebra of $B(\mc{K})$, then $\mc{M}\otimes\mc{N}$ is the von Neumann subalgebra of $B(\mc{H}\otimes\mc{K})$ generated by all $A\otimes B$ where $A\in\mc{M}$ and $B\in\mc{N}$. 
\end{defn}

We consider an infinite tensor product of Hilbert spaces. There are two paths, often called the complete and the incomplete (or Guichardet) tensor product. The name refers to the construction, both methods produce Hilbert spaces which are complete metric spaces. For details, see \cite{vN39,TW01,Parthasarathy92}. 

The {\em incomplete tensor product} assumes a unit vector $\xi_i$ is chosen for each $\mc{H}_i$ and takes the metric space completion of all pure tensors $\otimes_I v_i$ where $v_i=\xi_i$ for all but finitely many $i\in I$. The incomplete tensor is written $\bigotimes^\xi_I \mc{H}_i$. When applied to a countable family of separable Hilbert spaces it produces a separable Hilbert space and the construction is associative. 

The {\em complete tensor product} $\bigotimes_I\mc{H}_i$ does not require a choice of unit vectors. It is much larger. Even when applied to a countable family of separable Hilbert spaces, it is not separable. The details of the construction are involved, and won't be necessary for us. Roughly, one considers all tensors $\otimes_I v_n$ for which $\prod_I \| v_i\|^2$ is finite, then takes a metric space completion with respect to a certain inner product. While the complete tensor product is in general not associative, the following holds \cite[p.~43]{vN39}.

\begin{prop}\label{prop:associative}
For $i\in I$ we have $\mc{H}_i\otimes \bigotimes_{j\neq i}\mc{H}_j \simeq \bigotimes_I\mc{H}_i$. 
\end{prop}

The above ideas can be used to give notions of the tensor product of an infinite family of von Neumann algebras $\mc{M}_i$ where $\mc{M}_i$ acts on the Hilbert space $\mc{H}_i$ for $i\in I$. We won't require this, and leave the interested reader to see for example \cite{Nakagami70}.

\section{Quantum logic}

There are many versions of propositional ``quantum logic'' based in the original idea of Birkhoff and von Neumann \cite{BvN37} of interpreting terms in a projection lattice $P(\mc{H})$ or some generalization (see for example \cite{DGG04,Kalmbach83}). To give the bare bones of one approach, take an infinite set $x_1,x_2,\ldots$ of propositional variables and connectives $\wedge,\vee,\perp,0,1$. Terms $t$ formed in the usual recursive way are called propositions. 

\begin{defn}
Orthologic consists of the propositions $t$ such that the equation $t\approx 1$ is valid in every \ts{ol}; orthomodular logic consists of such $t$ with $t\approx 1$ valid in every \ts{oml}; and quantum logic is such $t$ valid in each projection lattice $P(\mc{H})$. 
\end{defn}

One issue is the lack of an implication in this treatment, and it appears that there is no adequate solution to this issue. The common approach is to view the {\em Sasaki hook} as implication. 

\begin{defn}
The Sasaki product is given by $x\cdot_s y = x\wedge(x^\perp\vee y)$ and the Sasaki hook by $x\to_s y = x^\perp\vee (x\wedge y)$. 
\end{defn}

While the Sasaki hook does not share all properties one would like of implication \cite{Kalmbach83}, it and the Sasaki product are residuated \cite{BJ72,Foulis60}, meaning that in any \ts{oml} we have 
\[x\cdot_s y\leq z \Leftrightarrow y\leq x\to_s z.\]

There are versions of predicate orthologic, orthomodular logic, and quantum logic developed using infinite joins and meets to interpret quantifiers. We refer to \cite{DGG04}. A different approach to predicate quantum logic, more tightly tied to projection lattices, is given in \cite{Weaver01}. In Section~\ref{section:logic} we view our quantum cylindric algebras through this lens. 

\section{Quantum monadic and cylindric algebras}

Adaptations of classical notions are straightforward. 

\begin{defn}
A monadic \ts{ol} $(L,\exists)$ is an \ts{ol} $L$ with unary operation $\exists$ where
\begin{itemize}[leftmargin=.4in]
\item[{\em \small (Q$_1$)}] $\exists 0 = 0$,
\item[{\em \small (Q$_2$)}] $p\leq \exists p$,
\item[{\em \small (Q$_3$)}] $\exists(p\vee q) = \exists p \vee \exists q$,
\item[{\em \small (Q$_4$)}] $\exists\exists p = \exists p$,
\item[{\em \small (Q$_5$)}] $\exists (\exists p)^\perp = (\exists p)^\perp$.
\end{itemize}
This is a quantum monadic algebra if $L$ is additionally an \ts{oml}.
\end{defn}

As in the classical case, there is an alternate way to view quantifiers and monadic \ts{ol}s. A subalgebra $S$ of an \ts{ol} $L$ is an {\em approximating subalgebra} if for each $a\in L$ there is a largest element $a^-\in S$ below $a$ and a least element $a^+\in S$ above $a$. When one of these exists, so does the other, and they are related by $a^-=((a^\perp)^+)^\perp$. A subalgebra $S$ of a complete \ts{ol} $L$ is a complete subalgebra if it is closed under arbitrary joins and meets as taken $L$. A complete subalgebra $S$ is approximating with $a^+=\bigwedge\{s\in S:a\leq s\}$. 

\begin{prop}
Let $L$ be an \ts{ol}. If $S\leq L$ is approximating, then $\exists a = a^+$ and $\forall a = a^-$ are existential and universal quantifiers. Conversely, if $\exists$ is an existential quantifier, then $S=\{\exists a:a\in L\}$ is an approximating subalgebra. 
\end{prop}

\begin{proof}
Suppose $S$ is approximating and $\exists a = a^+$ is the least element in $S$ above $a$. By definition {\small (Q$_2$)} holds, and {\small (Q$_1$), (Q$_4$), (Q$_5$)} hold since $0, \exists a, (\exists a)^\perp\in S$. Clearly $\exists a \vee \exists b\in S$ and lies above $a\vee b$. But any element of $S$ that lies above $a\vee b$ lies above $\exists a$ and $\exists b$, so lies above $\exists a\vee\exists b$. So $\exists (a\vee b)=\exists a\vee\exists b$, giving {\small (Q$_3$)}. So $\exists$ is a quantifier and by definition $\exists a $ is least in $S$ above $a$. From basic properties of orthocomplementation $\forall a = (\exists a^\perp)^\perp$ is largest in $S$ beneath $a$. Conversely, if $\exists$ is a quantifier on an \ts{ol} $L$, then {\small (Q$_1$), (Q$_3$), (Q$_5$)} give that $S=\{\exists a:a\in L\}$ is a subalgebra, and {\small (Q$_2$), (Q$_4$)} imply that $\exists a$ is least in $S$ above $a$, so $S$ is approximating. 
\end{proof}

The following observation is trivial, but useful. It requires only the definition of a pair of residuated maps and their well-known properties \cite{BJ72}. The point is that $\exists a \leq b$ and $a\leq\forall b$ are each equivalent to saying that there is an element $s$ of the associated approximating subalgebra with $a\leq s\leq b$. 

\begin{prop} \label{prop:quantifier residuated}
In a monadic \ts{ol}, $\exists a\leq b \Leftrightarrow a\leq\forall b$. Therefore $\exists$ preserves all existing joins and $\forall$ preserves all existing meets. 
\end{prop}

Monadic \ts{ol}s and quantum monadic algebras fall under the scope of universal algebra in the well-known sense \cite{BS81}. Thus there are standard notions of homomorphisms, subalgebras and products for these structures. It is easy to slightly reformulate their definitions to be purely equational, and thus both form varieties. We won't look closely at universal algebraic properties, but provide some basics. Here the difference between monadic \ts{ol}s and quantum monadic algebras is essential. 

\begin{prop}
Every quantum monadic algebra $(L,\exists)$ is congruence distributive and congruence permutable; and its congruences correspond to ideals $I$ of $L$ such that for all $a\in I$ and $b\in L$ we have $\exists a\in I$ and $b\wedge (a\vee b^\perp)\in I$. 
\end{prop}

\begin{proof}
Since \ts{oml}s are congruence distributive and permutable \cite{BH00} the same is true of quantum monadic algebras. Further, \cite{BH00} there is a correspondence between congruences $\theta$ of $L$ and ideals $I$ of $L$ that satisfy $a\in I$ and $b\in L$ imply $b\wedge (a\vee b^\perp)\in I$. Such ideals are called $p$-ideals. This associates to $\theta$ the $p$-ideal $0/\theta$ and associates to the $p$-ideal $I$ the congruence $\{(x,y):x\vee a = y\vee a\mbox{ for some }a\in I\}$. If $\theta$ is compatible with $\exists$, then since $\exists 0 = 0$ we have $I$ is closed under $\exists$. Conversely, if $I$ is closed under $\exists$ and $x\vee a = y\vee a$ for some $a\in I$, then $\exists x \vee \exists a = \exists(x\vee a)=\exists(y\vee a)=\exists y\vee\exists a$, so $(\exists x,\exists y)\in \theta$. 
\end{proof}

Elements $x,y$ of an \ts{oml} $L$ {\em commute} if $x=(x\wedge y)\vee(x\wedge y')$. Commutativity is symmetric in an \ts{oml}. An element $c$ is {\em central} if it commutes with all $x\in L$. An element $c$ is central iff ${\downarrow}\, c$ is a p-ideal. Central elements of an \ts{oml} correspond to its direct product decompositions. Since $\exists c = c$ iff $\exists c^\perp=c^\perp$, the direct product decompositions of a quantum monadic algebra correspond to central elements $c$ with $\exists c = c$. See \cite{BH00} for details about commuting elements in an \ts{oml}. 

\begin{prop}
If $(L,\exists)$ is a quantum monadic algebra and $a\in L$ with $a=\exists a$, then $C(a) = \{x\in L: a\mbox{ commutes with }x\}$ is a subalgebra. 
\end{prop}

\begin{proof}
For any $a\in L$, it is known \cite{BH00,Kalmbach83} that $C(a)$ is a sub-\ts{oml} of $L$. It remains to show that if $a$ commutes with $x$, then $a$ commutes with $\exists x$. Using that $a,x$ commute and that a quantifier preserves joins, 
$\exists x = \exists((x\wedge a)\vee(x\wedge a^\perp)) = \exists(x\wedge a) \vee\exists(x\wedge a^\perp)$. Since $\exists(x\wedge a)\leq\exists a = a$ and $\exists(x\wedge a^\perp) \leq\exists a^\perp = a^\perp$, it follows that from any three of the elements $a, a^\perp, \exists(x\wedge a)$ and $\exists(x\wedge a^\perp)$, one commutes with the other two, so by the Foulis-Holland theorem \cite{BH00}, the sublattice generated by these four elements is distributive. Then $(\exists x \vee a)\wedge(\exists x\vee a^\perp) = [\exists(x\wedge a^\perp)\vee a]\wedge[\exists(x\wedge a)\vee a^\perp]$ and using distributivity this is $\exists(x\wedge a)\vee\exists(x\wedge a^\perp)=\exists x$. So $\exists x$ commutes with $a$. 
\end{proof}

\begin{prop}\label{prop:interval}
If $(L,\exists)$ is a quantum monadic algebra and $a\in L$ with $\exists a = a$, then the interval $([0,a],\exists)$ is a quantum monadic algebra with orthocomplement $x^\# = a\wedge x^\perp$ and this quantum monadic algebra is a factor of the subalgebra $(C(a),\exists)$ of $(L,\exists)$. 
\end{prop}

\begin{proof}
Since $a$ is central in $C(a)$, it is well-known that $[0,a]$ and $[0,a^\perp]$ are \ts{oml}'s under the indicated orthocomplementations and $L\simeq [0,a]\times[0,a^\perp]$ via $x\mapsto (x\wedge a, x\wedge a^\perp)$. To see that $\exists$ is a quantifier on $[0,a]$ only {\small (Q$_5$)} is non-trivial. But If $x\leq a$, then $(\exists x)^\# = a\wedge(\exists x)^\perp$ is a meet of elements in the approximating subalgebra, so is in the approximating subalgebra, giving {\small (Q$_5$)}. But $\exists x$ commuting with $a$ gives $\exists x = \exists(x\wedge a)\vee\exists(x\wedge a^\perp)$, so $\exists x \wedge a = \exists(x\wedge a)$ and $\exists x\wedge a^\perp = \exists(x\wedge a^\perp)$, showing $x\mapsto (x\wedge a, x\wedge a^\perp)$ is a quantum monadic algebra isomorphism. 
\end{proof}

\begin{example} {\em 
If $L$ is an \ts{ol} and $S$ is a finite subalgebra of $L$, then $S$ is an approximating subalgebra, so yields a monadic \ts{ol}. The diagram below shows a finite \ts{oml} and a subalgebra indicated by filled in dots. This happens to be a Boolean subalgebra. 

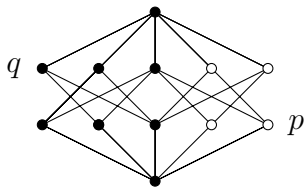
\begin{center}
\begin{tikzpicture}[scale=.75]
\draw (0,0) -- (-2,1) -- (-1,2) -- (0,1) -- (-2,2) -- (-1,1) -- (0,2) -- (0,3) -- (-2,2) -- (0,3) -- (-1,2) -- (-2,1) -- (0,0) -- (-1,1) -- (0,0) -- (0,1) -- (1,2) -- (0,3) -- (2,2) -- (1,1) -- (0,0) -- (2,1) -- (0,2) -- (-2,1) -- (0,3) -- (2,2) -- (0,1) -- (0,0) -- (2,1) -- (1,2) -- (0,3) -- (0,2) -- (1,1);
\draw[fill] (0,0) circle (2.5pt); \draw[fill] (-1,1) circle (2.5pt); \draw[fill] (-1,1) circle (2.5pt); \draw[fill] (0,1) circle (2.5pt);
\draw[fill] (0,2) circle (2.5pt); \draw[fill] (-1,2) circle (2.5pt);  \draw[fill] (-2,2) circle (2.5pt);  \draw[fill] (0,3) circle (2.5pt); 
\draw[fill] (-2,1) circle (2.5pt); \draw[fill=white] (1,1) circle (2.5pt); \draw[fill=white] (2,1) circle (2.5pt); \draw[fill=white] (1,2) circle (2.5pt); \draw[fill=white] (2,2) circle (2.5pt); 
\node at (2.5,1) {$p$}; \node at (-2.5,2) {$q$}; 
\end{tikzpicture}
\end{center}
For each element $x$, the smallest filled in dot above $x$ is $\exists x$. In this case, $\exists (p\wedge\exists q) = 0$ while $\exists p \wedge \exists q$ is an atom. This provides an example of a quantum monadic algebra that does not satisfy {\small (Q$_6$)} of Definition~\ref{defn:monadic}. 
}
\end{example}

\begin{example} {\em 
As mentioned, a complete subalgebra of a complete \ts{ol} produces a monadic \ts{ol}. A {\em block} of an \ts{oml} is a maximal Boolean subalgebra. It is known that a block $B$ of $L$ is closed under all existing joins and meets in $L$. So if $L$ is complete, a block is a complete subalgebra, hence produces a quantum monadic algebra. Completeness here is required. There is a method to construct an \ts{oml} $K(L)$ from a bounded lattice $L$ \cite{Harding91,Kalmbach83}. Let $L$ be the bounded lattice $([0,1)\times 2) \oplus (0,1]$, where $[0,1)$ and $(0,1]$ are intervals in the reals and $\oplus$ is ordinal sum. For $x$ the element $\langle 0 < (0,1)\rangle$ in $K(L)$ and $B$ the block of $K(L)$ corresponding to the maximal chain $C=([0,1)\times \{0\})\oplus (0,1]$ of $L$, there is no least element in this block above $x$. See \cite{Harding91} for details. 
}
\end{example}

\begin{example} \label{example:vN} {\em 
Let $\mc{M}$ be a von Neumann subalgebra of the bounded operators $\mc{B}(\mc{H})$ of a Hilbert space $\mc{H}$. The projections $P(\mc{M})$ are a complete subalgebra of the complete \ts{oml} $P(\mc{H})$ of projection operators of $\mc{H}$. So this provides a quantum monadic algebra. The underlying \ts{oml} of this quantum monadic algebra is $P(\mc{H})$ and for any $p\in P(\mc{H})$ we have $\exists p$ is the least projection in $\mc{M}$ above $p$. 
}
\end{example}

\begin{example} {\em 
Let $L$ be a complete \ts{oml}. Then its center, the elements that commute with all others, is a complete subalgebra \cite{BH00,Kalmbach83}. Therefore, this provides a quantum monadic algebra. When applied to the complete \ts{oml} $P(\mc{M})$ of projections of a von Neumann algebra $\mc{M}$, the least central element $\exists p$ above a projection $p\in\mc{M}$ is commonly called the central carrier of $p$. 
}
\end{example}

\begin{example} {\em 
Let $\mc{M}$ be a von Neumann subalgebra of the bounded operators $\mc{B}(\mc{H})$ of a Hilbert space $\mc{H}$. For a projection $p\in \mc{M}$, the {\em corner} $p\mc{M}p$ \cite{Blackadar06} is $*$-isomorphic to a von Neumann subalgebra of $\mc{B}(p\mc{H})$. As they are von Neumann algebras, $P(\mc{M})$ provides a quantum monadic structure on $P(\mc{H})$ and $P(p\mc{M}p)$ such a structure on $P(p\mc{H})$. In fact, this is an instance of Proposition~\ref{prop:interval} with the second an interval of the first. 
}
\end{example}

\begin{example} \label{ex:subfactor} {\em 
For $\mc{N}$ a von Neumann subalgebra of a von Neumann algebra $\mc{M}$ the projections $P(\mc{N})$ are a complete subalgebra of the complete \ts{oml} $P(\mc{M})$ and therefore provides an example of a quantum monadic algebra. This applies, in particular, when both $\mc{M}$ and $\mc{N}$ are factors, a situation called a {\em subfactor}. Here there is an extensive literature \cite{GPJ89,JS97} including an associated {\em Jones tower} of subfactors $\mc{N}\leq\mc{M}\leq \cdots$ that yields a corresponding tower $(L_1,\exists_{\mc N})\leq (L_2,\exists_\mc{M}) \leq\cdots$ of quantum monadic algebras. 
}
\end{example}

The motivating example of a monadic algebra is the \ts{ba} $B^X$ of all functions from a set $X$ into a complete \ts{ba} $B$ with $\exists f$ being the constant function taking value $\bigvee\{f(x):x\in X\}$. A monadic algebra is {\em functional} if it is a subalgebra of such. Halmos \cite{Halmos62} showed that each monadic algebra is isomorphic to a functional one. The situation for monadic \ts{ol}s and quantum monadic algebras is open, but we make some comments.

\begin{remark} {\em
There is a notion of a monadic Heyting algebra \cite{MV57} extending Halmos' monadic (Boolean) algebras. In \cite{BH02} it was shown that every monadic Heyting algebra is isomorphic to a functional one. The method of proof was different from Halmos' result, using instead a construction involving strong amalgamations and MacNeille completions. These tools are available for \ts{ol}s, so this is a natural direction to approach the question. 

Every \ts{oml} is the underlying \ts{ol} of a quantum monadic algebra, one simply takes the subalgebra $\{0,1\}$. It is a long-open question whether every \ts{oml} can be embedded into a complete \ts{oml} \cite{Harding91,Harding98}, so it would be highly problematic to show that every quantum monadic algebra can be embedded into a functional one. One might consider the question for complete quantum monadic algebras whose approximating subalgebra is Boolean. Here there is a slight opening provided by the result that \ts{oml}s can be amalgamated over Boolean subalgebras \cite{BH02}, but this situation is far from clear. 
}
\end{remark}

\begin{defn}
A weak $I$-dimensional cylindric \ts{ol} $(L,c_i,d_{i,j})$ is an \ts{ol} $L$ with family of unary operators $c_i$ and constants $d_{i,j}$ where $i,j\in I$ such that for all $i,j,k\in I$
\begin{itemize}[leftmargin=.4in]
\item[{\em \small (C$_1$)}] $c_i$ is a quantifier,
\item[{\em \small (C$_2$)}] $c_i c_j x = c_j c_i x$,
\item[{\em \small (C$_3$)}] $d_{i,j}=d_{j,i}$ and $d_{i,i}=1$,
\item[{\em \small (C$_4$)}] if $j\neq i,k$ then $d_{i,k}=c_j(d_{i,j}\wedge d_{j,k})$,
\end{itemize}
This is an $I$-dimensional cylindric \ts{ol} if it additionally satisfies
\begin{itemize}[leftmargin=.4in]
\item[{\em \small (C$_5$)}] if $i\neq j$ then $c_i(d_{i,j}\wedge x)\wedge c_i(d_{i,j}\wedge x^\perp) = 0$.
\end{itemize}
Quantum weak cylindric algebras are those where the underlying \ts{ol} is an \ts{oml}. 
\end{defn}

The motivating example of a classical cylindric algebra is a cylindric set algebra described in Example~\ref{ex:cylindric set algebra}. In Section~\ref{section:logic} we discuss a quantum analog of a cylindric set algebra that informs our choice of axioms. We briefly discuss the axioms below. 

\begin{remark}\label{rem:substitution} {\em 
The purpose of axiom {\small (C$_5$)} in the classical setting, at least in part, is to provide that the substitution operators $S_j^{\, i}x = c_i(d_{i,j}\wedge x)$ of Definition~\ref{defn:substitution} are Boolean endomorphisms. Indeed, {\small (C$_5$)} says that $S_j^{\, i} x \wedge S_j^{\, i}x^\perp = 0$. Since $c_i$ and the operation $d_{i,j}\wedge (\,\cdot\,)$ preserve joins in the classical setting, $S_j^{\, i}$ preserves joins, so {\small (C$_5$)} provides that it also preserves complements, and therefore is a Boolean endomorphism. As we will see, {\small (C$_5$)} does not hold in our quantum analog of cylindric set algebras. But even if it did, the utility would be limited --- in an \ts{oml} the operation $d_{i,j}\wedge(\,\cdot\,)$ does not preserve joins, and knowing that $x,y$ are complements does not provide that they are orthocomplements. More problematic is that in primary examples of \ts{oml}s, projections of a Hilbert space, there are often limited supplies of endomorphisms. The story here is not settled, but we remark that the operation $S_j^{\, i} x = c_i(d_{i,j}\wedge(d_{i,j}^\perp\vee x))$ reduces to the standard substitution operator in the classical case, and since quantifiers and Sasaki products are residuated, this operation preserves joins. 
}
\end{remark}

The following section, Section~\ref{section:logic}, gives in detail a quantum version of a cylindric set algebra. We conclude this section with a discussion of a different sort of example. Here, the situation is not fully settled, but in the simplest setting there is a link between quantum cylindric algebras and what are called commuting squares of subfactors. For this discussion, we need some additional background on von Neumann algebras. 

\begin{defn}
An element $x$ in a von Neumann algebra $\mc{M}$ is positive if $x=a^*a$ for some $a\in\mc{M}$. There is a partial ordering $\leq_{sa}$ on the self-adjoint elements given by $x\leq y$ iff $y-x$ is positive. The positive elements of $\mc{M}$ lying in the interval $[0,1]$ of $\mc{M}$ are called effects. 
\end{defn}

Trivially every projection is positive, and since $1-p$ is the orthocomplementary projection to $p$, every projection is an effect. The partial ordering $\leq_{sa}$ when restricted to projections is the usual ordering of projections. We will customarily write $\leq_{sa}$ simply as $\leq$ since it will not cause confusion. The key to finer properties relating to the ordering of effects and projections is what is known as the {\em range projection} $P(x)$. The following is well-known, and is found for example in \cite[Lemma~3.1]{CH02} and \cite[p.~56]{Westerbaan18}.

\begin{prop}
If $x$ is an effect of $\mc{B}(\mc{H})$, then there is a least projection $P(x)$ of $\mc{H}$ with $x\leq P(x)$. This $P(x)$ is given by $P(x)={\displaystyle \lim_{n\to\infty}x^{1/n}}$ and belongs to the von Neumann subalgebra of $\mc{B}(\mc{H})$ that is generated by $x$. 
\end{prop}

If $x$ is an effect of $\mc{B}(\mc{H})$, then $P(x)$ is the least projection of $\mc{H}$ above $x$. If $\mc{M}$ is a von Neumann subalgebra of $\mc{B}(\mc{H})$, as we noted in Example~\ref{example:vN}, there is a least projection in $\mc{M}$ above the projection $P(x)$, hence a least projection in $\mc{M}$ above $x$.

\begin{defn}
If $x$ is an effect of $\mc{B}(\mc{H})$ and $\mc{M}$ is a von Neumann subalgebra of~$\mc{B}(\mc{H})$, let $\exists_\mc{M}x$ be the least projection in $\mc{M}$ above $x$. 
\end{defn}

Note that $\exists_\mc{M}x = \exists_\mc{M}P(x)$ and that if $x\in\mc{M}$, then $\exists_\mc{M}x=P(x)$ since $P(x)$ belongs to the von Neumann subalgebra generated by $x$ and hence belongs to $\mc{M}$. If we restrict $\exists_\mc{M}$ to apply to projections of $\mc{H}$, then we obtain the quantum monadic algebra of Example~\ref{example:vN}. This extends further. If $\mc{M}$ is a von Neumann subalgebra of $\mc{B}(\mc{H})$ and $\mc{N}\leq\mc{M}$ is a von Neumann subalgebra, then consider the restriction of $\exists_\mc{N}$ to $\mc{M}$. If this $\exists_\mc{N}$ is further restricted to the projections of $\mc{M}$ we obtain the situation in Example~\ref{ex:subfactor}. We will not distinguish between these various restrictions, and leave it to context. 

\begin{defn}
If $\mc{N}$ is a von Neumann subalgebra of $\mc{M}$, a positive linear map $E:\mc{M}\to\mc{N}$ is a conditional expectation if $E(1)=1$ and $E(nmn') = nE(m)n'$ for all $n,n'\in\mc{N}$ and $m\in\mc{M}$. 
\end{defn}

For background on conditional expectations see  \cite[p.~147]{Kadison04}, \cite[Sec.~9]{AP10}, \cite[p.~32.]{Speicher16}, and \cite{JS97,KR97b}. Note that the defining conditions of a conditional expectation imply that $En=n$ for each $n\in\mc{N}$, so a conditional expectation is idempotent. The following is well-known and found for example in \cite[Thm.~7, Prop.~15]{Kadison04}. 

\begin{thm}
If $\mc{M}$ is a von Neumann algebra with faithful tracial state $\tau$ and $\mc{N}$ is a von Neumann subalgebra of $\mc{M}$, then there is a unique conditional expectation $E_\mc{N}:\mc{M}\to\mc{N}$ with $\tau E_\mc{N}=\tau$. 
\end{thm}

The following is result of Pimsner and Popa \cite[Prop.~2.1]{PP86} is of importance (see also \cite[p.~8]{Naaijkens17}). To state this result, we mention the famous Jones index \cite{JS97}. While this provides insight to those familiar with this notion, a knowledge of this index is not important for purposes here, it is enough to know that in the our setting this index will be a real number with value at least 1. 

\begin{thm}\label{thm:index}
Let $\mc{M}$ be a type II$_1$ factor, $\mc{N}\leq\mc{M}$ be a subfactor of $\mc{M}$ of finite index $k=[\mc{M}:\mc{N}]$, and $E_\mc{N}$ be the unique conditional expectation compatible with the trace of $\mc{M}$. Then $\frac{1}{k}x\leq E_\mc{N}\,x$ for all positive $x\in\mc{M}$. 
\end{thm}

This setting of a type II$_1$ factor and a subfactor of finite index is the important setting for much of Jones' work related to planar algebras and knots \cite{JS97}. In this setting, we relate conditional expectations to quantifiers in the following. 

\begin{prop} \label{prop:quantifier-expectation}
If $\mc{M}$ is a type II$_1$ factor and $\mc{N}\leq\mc{M}$ is a subfactor of finite index, then $\exists_\mc{N}\,p=\exists_\mc{N} E_\mc{N}\, p = P(E_\mc{N}\,p)$ for each projection $p$ of $\mc{M}$. 
\end{prop}

\begin{proof}
Note first that $E_\mc{N}\,p$ is an effect since $0\leq p\leq 1$ and the conditional expectation $E_\mc{N}$ is positive, hence order preserving, and $E_\mc{N}\,1=1$. By Theorem~\ref{thm:index} there is a positive $\lambda\leq 1$ with $\lambda p \leq E_\mc{N}\,p$. Since $P(x)={\displaystyle \lim_{n\to\infty}x^{1/n}}$ for each effect $x$, it follows that $p=P(\lambda p)\leq P(E_\mc{N}\,p)$. But $E_\mc{N}\,p\in\mc{N}$ implies that the least projection $P(E_\mc{N}\,p)$ above it belongs to $\mc{N}$, and therefore $P(E_\mc{N}\,p)$ is the least projection in $\mc{N}$ above $E_\mc{N}\,p$. This gives $p\leq P(E_\mc{N}\,p)=\exists_\mc{N}E_\mc{N}\,p$. 
Then as $\exists_\mc{N}E_\mc{N}\,p$ is a projection in $\mc{N}$ above $p$, we have $\exists_\mc{N}\,p\leq \exists_\mc{N}E_\mc{N}\,p$. For the other inequality, set $q=\exists_\mc{N}\,p$. Then $p\leq q$ gives $\exists_\mc{N}E_\mc{N}\,p\leq\exists_\mc{N}E_\mc{N}\,q=\exists_\mc{N}\,q$ since $E_\mc{N}$ is the identity on $\mc{N}$. But $\exists_\mc{N}\,q=q$ since $q$ is a projection in $\mc{N}$. Then $q=\exists_\mc{N}\,p$  gives $\exists_\mc{N}E_\mc{N}\,p\leq\exists_\mc{N}\,p$. 
\end{proof}

\begin{prop}\label{prop:commuting quantifiers}
If $\mc{L}$ is a type II$_1$ factor and $\mc{M},\mc{N}$ are subfactors of $\mc{L}$ of finite index such that the conditional expectations $E_\mc{M}:\mc{L}\to\mc{M}$ and $E_\mc{N}:\mc{L}\to\mc{N}$ commute, then the quantifiers $\exists_\mc{M}:P(\mc{L})\to P(\mc{M})$ and $\exists_\mc{N}:P(\mc{L})\to P(\mc{N})$ commute.
\end{prop}

\begin{proof}
We first note that $\exists_\mc{M}$ and $\exists_\mc{N}$ commute, iff (1) $q\in P(\mc{M})$ implies $\exists_\mc{N}\,q\in\mc{M}\cap\mc{N}$ and (2) $r\in P(\mc{N})$ implies $\exists_\mc{M}\,r\in\mc{M}\cap\mc{N}$. Indeed, if the quantifiers commute and $q\in P(\mc{M})$, then $\exists_\mc{N}\,q=\exists_\mc{N}\exists_\mc{M}\,q = \exists_\mc{M}\exists_\mc{N}\,q$ hence belongs to both $\mc{M}$ and $\mc{N}$. This gives (1) and (2) is similar. Conversely, suppose (1) and (2) hold. Then for $p\in P(\mc{L})$, setting $q=\exists_\mc{M}\,p$ we have $\exists_\mc{N}\,q$ is a projection in $\mc{M}\cap\mc{N}$ that lies above $p$ and hence lies above $\exists_\mc{N}\,p$ and therefore also above $\exists_\mc{M}\exists_\mc{N}\,p$. So $\exists_\mc{M}\exists_\mc{N}\,p\leq\exists_\mc{N}\exists_\mc{M}\,p$ and the other inequality follows from using condition (2) instead of (1). Assume the conditional expectations $E_\mc{M}$ and $E_\mc{N}$ commute and $q\in P(\mc{M})$. By Proposition~\ref{prop:quantifier-expectation}, $\exists_\mc{N}\,q = P(E_\mc{N}\,q)$. Since $q\in\mc{M}$ we have $q=E_\mc{M}\,q$, and then as conditional expectations commute $E_\mc{N}\,q = E_\mc{N}E_\mc{M}\,q=E_\mc{M}E_\mc{N}\,q$, giving that $E_\mc{N}\,q\in\mc{M}$. So $P(E_\mc{M}\,q)\in\mc{M}$, showing that $\exists_\mc{N}\exists_\mc{M}\,q\in\mc{M}$. This establishes (1) and (2) is similar. So $\exists_\mc{M}$ and $\exists_\mc{N}$ commute. 
\end{proof}

\begin{example} {\em 
Originating with a non-commutative version of independent $\sigma$-fields \cite{Popa83,WW95}, the notion of a {\em commuting square} of subfactors plays an important role in their theory \cite[p.~188]{GPJ89}. Suppose $\mc{L}$ is a type II$_1$ factor with a finite normal tracial state $\tau$ and $\mc{M}$ and $\mc{N}$ subfactors of $\mc{L}$. If $\mc{M}\cap\mc{N}=\mc{K}$, then $\mc{K},\mc{M},\mc{N},\mc{L}$ is a commuting square of subfactors if the associated conditional expectations $E_\mc{M}:\mc{L}\to\mc{M}$ and $E_\mc{N}:\mc{L}\to\mc{N}$ commute. 
\begin{center}
\begin{tikzpicture}[scale=1]
\node at (0,0) {$\mc{K}$}; \node at (2,0) {$\mc{N}$}; \node at (0,1.4) {$\mc{M}$}; \node at (2,1.4) {$\mc{L}$}; 
\node at (1,0) {$\subseteq$}; \node at (1,1.3) {$\subseteq$};
\node at (0,.65) {\rotatebox[origin=c]{90}{$\subseteq$}}; \node at (2,.65) {\rotatebox[origin=c]{90}{$\subseteq$}};
\end{tikzpicture}
\end{center}
Such a commuting square gives a an \ts{oml} $P(\mc{L})$ with two quantifiers, and by Proposition~\ref{prop:commuting quantifiers} these quantifiers commute. Thus $(P(\mc{L}),\exists_\mc{M},\exists_\mc{N})$ is a 2-dimensional diagonal-free quantum cylindric algebra. 
}
\end{example}

%
%

\section{Quantum cylindric set algebras and quantum predicate calculus}\label{section:logic}

The motivating example of a cylindric algebra is that of a cylindric set algebra. This is the powerset of a product $\prod_IX_i$ of sets with cylindrification $c_i \,S$ being all choice functions $\alpha$ that agree with some member of $S$ except possibly in the $i^{th}$-coordinate. If all factors of this product are equal, so we have a power $X^I$, diagonals are defined as $d_{i,j} =\{\alpha:\alpha(i)=\alpha(j)\}$. We provide a quantum analog using tensor products and powers rather than Cartesian product. Our treatment shares features with Weaver's approach to quantum predicate logic \cite{Weaver01}. 

\begin{defn}
For a Hilbert space $\mc{H}$ let $\mc{C}(\mc{H})$ be its \ts{oml} of closed subspaces. 
\end{defn}

While our eventual purpose is to work in a tensor product $\mc{H}_1\otimes\cdots\otimes\mc{H}_n$, or even an infinite such tensor product, our considerations involve the relation of one or two factors to the whole. So it is more convenient to work with a binary or ternary tensor product and use associativity. To this end, for $A, B$ closed subspaces of $\mc{H}$ and $\mc{K}$, respectively, note that $A\otimes B$ is the closed subspace of $\mc{H}\otimes\mc{K}$ generated by all $a\otimes b$ where $a\in A$ and $b\in B$, and we use $\mc{H}\otimes\mc{C}(\mc{K})$ for the collection of all closed subspaces of $\mc{H}\otimes\mc{K}$ of the form $\mc{H}\otimes B$ where $B\in\mc{C}(\mc{K})$. 

\begin{prop}\label{prop:iso}
The map $\alpha:\mc{C}(\mc{K})\to\mc{C}(\mc{H}\otimes \mc{K})$ where $\alpha(B)=\mc{H}\otimes B$ is a complete \ts{ol}-embedding and its image is $\mc{H}\otimes\mc{C}(\mc{K})$. 
\end{prop}

\begin{proof} 
This map is well-defined, order preserving, and onto. Let $B\in\mc{C}(\mc{K})$ and take \ts{onb}'s $(e_i)_I$ of $\mc{H}$, $(f_j)_J$ of $B$ and $(g_k)_K$ of $B^\perp$. Then $(e_i\otimes f_j)_{I\times J}$ is an \ts{onb} of $\mc{H}\otimes B$ and $(e_i\otimes g_k)_{I\times K}$ of $\mc{H}\otimes B^\perp$. Thus $\mc{H}\otimes B$ and $\mc{H}\otimes B^\perp$ are orthocomplements. Suppose $\bigvee_I B_i = B$. As $\alpha$ is order preserving $\bigvee_I \mc{H}\otimes B_i\leq \mc{H}\otimes B$. For $v\in B$ there is a member of the linear span of $\bigcup_I B_i$ arbitrarily close to $v$, so for any $u\in\mc{H}$ there are members of the linear span of $\bigcup_I\mc{H}\otimes B_i$ arbitrarily close to $u\otimes v$. It follows that $\mc{H}\otimes B\leq\bigvee_I\mc{H}\otimes B_i$. This shows that $\alpha$ is a complete \ts{ol}-homomorphism onto $\mc{H}\otimes\mc{C}(\mc{K})$. It is easily seen that $\alpha(B)=0$ iff $B=0$. Since $\alpha$ is a homomorphism between \ts{oml}s, it follows that $\alpha$ is an embedding. 
\end{proof}

Since $\mc{H}\otimes\mc{C}(\mc{K})$ is a complete subalgebra of $\mc{C}(\mc{H}\otimes\mc{K})$ we may define the following. 

\begin{defn}
For $S\in\mc{C}(\mc{H}\otimes\mc{K})$ let 
\begin{eqnarray*}
\exists_{\mc H} S&\mbox{ be the smallest in $\mc{H}\otimes\mc{C}(\mc{K})$ above $S$,}\\
\forall_{\mc H} S&\mbox{ be the largest in $\mc{H}\otimes\mc{C}(\mc{K})$ below $S$.}
\end{eqnarray*}
\end{defn} 

We provide a more explicit description of these operations. 

\begin{prop}
If $S\in\mc{C}(\mc{H}\otimes\mc{K})$, then $\forall_\mc{H}S = \mc{H}\otimes S_\mc{H}$ where 
\[S_\mc{H}=\{w\in\mc{K} : a\otimes w\in S\mbox{ for all }a\in\mc{H}\}.\]
\end{prop}

\begin{proof}
It is easily seen that $S_\mc{H}\in\mc{C}(\mc{K})$. If $w\in S_\mc{H}$, then $a\otimes w\in S$ for all $a\in\mc{H}$, so a generating set of $\mc{H}\otimes S_\mc{H}$ is contained in $S$, giving $\mc{H}\otimes S_\mc{H}\leq S$. Suppose $T\in\mc{C}(\mc{K})$ and $\mc{H}\otimes T\leq S$. Then for $w\in T$ we have $a\otimes w\in S$ for all $a\in\mc{H}$, so $T\leq S_\mc{H}$. Thus $\mc{H}\otimes S_\mc{H}$ is largest in $\mc{H}\otimes\mc{C}(\mc{K})$ beneath $S$. 
\end{proof}

Of course, $\exists_\mc{H}$ is given by $\perp\forall_\mc{H}\perp$, but we can give an explicit description. For this, note that for an \ts{onb} $(a_i)_I$ of $\mc{H}$, each element $v\in\mc{H}\otimes\mc{K}$ can be uniquely expressed as $\sum_I a_i\otimes v_i$ for some family $(v_i)_I\in\mc{K}$. This family depends on the choice of \ts{onb} $(a_i)_I$. We use $\langle T\rangle$ for the closed subspace generated by a set $T$. 

\begin{defn}
Let $(a_i)_I$ be an \ts{onb} of $\mc{H}$ and $S\subseteq \mc{H}\otimes\mc{K}$. Set
\[S^{\mc H} = \langle v_i : v\in S \mbox{ and } i\in I \rangle.\]
\end{defn}

\begin{lemma}\label{lem:tech}
For $(a_i)_I$ an \ts{onb} of $\mc{H}$ and $S\subseteq \mc{H}\otimes\mc{K}$ we have $\langle S\rangle^{\mc H} = S^{\mc H}$. 
\end{lemma}

\begin{proof}
One containment is obvious. For the other, assume $w\in\langle S\rangle$. Then for $\epsilon >0$ there is $v=\lambda_1v^1+\cdots+\lambda_nv^n$ in the linear span of $S$ with $\|w-v\|^2<\epsilon$. Since $(a_i)_I$ is an \ts{onb} and $w-v=\sum_Ia_i\otimes(w_i-v_i)$ 
\[ \sum_I\|a_i\otimes(w_i-v_i)\|^2 = \| \sum_I a_i\otimes(w_i-v_i)\|^2 = \| w-v\|^2 <\epsilon\]
As $\| a_i\otimes(w_i-v_i)\|^2 = \|a_i\|^2\|w_i-v_i\|^2=\|w_i-v_i\|^2$, for any $i_0\in I$ we have 
\[ \|w_{i_0}-(\lambda_1v_{i_0}^1+\cdots+\lambda_nv_{i_0}^n)\|^2 = \| w_{i_0}-v_{i_0}\|^2 < \epsilon \]
So $w_{i_0}\in S^{\mc H}$, and it follows that $\langle S\rangle^{\mc H} \subseteq S^{\mc H}$. 
\end{proof}

\begin{prop}
If $(a_i)_I$ an \ts{onb} of $\mc{H}$ and $S\in \mc{C}(\mc{H}\otimes\mc{K})$, then $\exists_{\mc H} S = \mc{H}\otimes S^\mc{H}$.
\end{prop}

\begin{proof}
By definition, $\mc{H}\otimes S^\mc{H}$ belongs to $\mc{H}\otimes\mc{C}(\mc{K})$. Let $v\in S$. Then $a_i\otimes v_i\in\mc{H}\otimes S^\mc{H}$ for each $i\in I$, and as $v=\sum_Ia_i\otimes v_i$ we have $v\in \mc{H}\otimes S^\mc{H}$. So $S\leq \mc{H}\otimes S^\mc{H}$. It remains to show that this is the least member of $\mc{H}\otimes\mc{C}(\mc{K})$ that contains $S$. Suppose $T\in\mc{C}(\mc{K})$ and $S\leq \mc{H}\otimes T$. Since $H\otimes T = \langle a_i\otimes b:i\in I, b\in T\rangle$, by Lemma~\ref{lem:tech} $(\mc{H}\otimes T)^\mc{H} = \langle (a_i\otimes b)_j:i,j\in I, b\in T\rangle$. Thus $S^\mc{H}\leq (\mc{H}\otimes T)^\mc{H} = T$. So $\mc{H}\otimes S^\mc{H}\leq \mc{H}\otimes T$ as required. 
\end{proof}

\begin{cor}
The definition of $S^\mc{H}$ is independent of the choice of \ts{onb} of $\mc{H}$. 
\end{cor}

\begin{proof}
$\exists_\mc{H}S = \mc{H}\otimes S^\mc{H}$ and by Proposition~\ref{prop:iso} if $\mc{H}\otimes T = \mc{H}\otimes T'$ then $T=T'$. 
\end{proof}

Consider a tensor product of three Hilbert spaces $\mc{H}\otimes\mc{K}\otimes\mc{M}$. Since tensor products are commutative and associative, we can consider this to be a binary tensor product in various ways, so our notation extends to allow $\exists_{\mc H}, \exists_{\mc{H}\otimes \mc{K}}, S^\mc{H}, S^{\mc{H}\otimes\mc{K}}$ and so forth. 

\begin{prop}
For $S\in\mc{C}(\mc{H}\otimes\mc{K}\otimes\mc{M})$, we have $\exists_\mc{H}\exists_\mc{K}S = \exists_{\mc{H}\otimes\mc{K}}S=\exists_\mc{K}\exists_\mc{H}S$. 
\end{prop}

\begin{proof}
Since $\exists_{\mc{H}\otimes\mc{K}}S = \mc{H}\otimes\mc{K}\otimes T$ for some $T\in\mc{C}(\mc{K})$, it belongs to both $\mc{H}\otimes\mc{C}(\mc{K}\otimes\mc{M})$ and $\mc{K}\otimes\mc{C}(\mc{H}\otimes\mc{M})$. So $\exists_\mc{K}S$ lies beneath $\exists_{\mc{H}\otimes\mc{K}}S$, hence $\exists_\mc{H}\exists_\mc{K}S$ lies beneath $\exists_{\mc{H} \otimes\mc{K}}S$. 

Let $(a_i)_I$ and $(b_j)_J$ be \ts{onb}'s of $\mc{H}$ and $\mc{K}$. Then $(a_i\otimes b_j)_{I\times J}$ is an \ts{onb} of $\mc{H}\otimes\mc{K}$. Each $v\in\mc{H}\otimes\mc{K}\otimes\mc{M}$ has unique representations 
\[v=\sum_{I\times J} a_i\otimes b_j\otimes v_{i,j}\quad\mbox{ and }\quad v=\sum_Jb_j\otimes v_j\]
where $(v_{i,j})_{I\times J}$ is a family in $\mc{M}$ and $(v_j)_J$ is a family in $\mc{H}\otimes\mc{M}$. Uniqueness gives $v_j=\sum_I a_i\otimes v_{i,j}$. To show that $\exists_{\mc{H}\otimes\mc{K}}S\leq\exists_\mc{H}\exists_\mc{K}S$, it is sufficient to show that if $v\in S$ and $i_0,i_1\in I$, $j_0,j_1\in J$, then $a_{i_1}\otimes b_{j_1}\otimes v_{i_0,j_0}\in\exists_\mc{H}\exists_\mc{K}S$ since such elements generate the closed subspace $\exists_{\mc{H}\otimes\mc{K}}S$. As $v\in S$ we have $v_{j_0}\in S^\mc{K}$, and therefore $b_{j_1}\otimes v_{j_0}\in\exists_\mc{K}S$. But 
\[ b_{j_1}\otimes v_{j_0} = \sum_I a_i\otimes b_{j_1}\otimes v_{i,j_0}.\]
Since this element is in $\exists_\mc{K}S$, then $(b_{j_1}\otimes v_{j_0})_{i_0} = b_{j_1}\otimes v_{i_0,j_0}$ belongs to $(\exists_\mc{K}S)^\mc{H}$. Therefore $a_{i_1}\otimes b_{j_1}\otimes v_{i_0,j_0}$ belongs to $\exists_\mc{H}\exists_\mc{K}S$. This provides the equality of $\exists_\mc{H}\exists_\mc{K}S$ and $\exists_{\mc{H}\otimes\mc{K}}S$, the equality of $\exists_\mc{K}\exists_\mc{H}S$ and $\exists_{\mc{H}\otimes\mc{K}}S$ is by symmetry. 
\end{proof}

Using the associativity of a finite tensor product $\mc{H}_1\otimes\cdots\otimes\mc{H}_n$, or the fragment of associativity for a complete infinite tensor product $\bigotimes_I\mc{H}_i$ given in Proposition~\ref{prop:associative} we have the following. 

\begin{thm}
For a family $(\mc{H}_i)_I$ of Hilbert spaces, the operations $(\exists_{\mc{H}_i})_I$ are pairwise commuting quantifiers on the \ts{oml} $\mc{C}(\bigotimes_I\mc{H}_i)$ of closed subspaces of the complete tensor product. So $\mc{C}(\bigotimes_I\mc{H}_i)$ is an $I$-dimensional diagonal-free quantum cylindric algebra. 
\end{thm}

\begin{example} {\em 
Let $\mc{L}$ be a first-order language consisting of a family $\mc{R}$ of relation symbols, a set  $(x_i)_I$ of variables, and connectives $\wedge,\vee,\perp,0,1,\exists_{x_i},\forall_{x_i}$. To each basic formula $r(x_{i_1},\ldots,x_{i_n})$ we associate a closed subspace of $\bigotimes_I\mc{H}_i$. Interpreting the connectives $\wedge,\vee,\perp,0,1, \exists_{x_i}, \forall_{x_i}$ as meet, join, orthocomplementation, 0, 1, $\exists_{\mc{H}_i}$ and $\forall_{\mc{H}_i}$ gives a subalgebra of the diagonal-free quantum cylindric algebra described above. 
}
\end{example}

We next consider diagonal elements. Here we require a tensor product $\mc{H}_1\otimes\cdots\otimes\mc{H}_n$ with all factors equal to the same Hilbert space $\mc{H}$. We often write this $\mc{H}^{\otimes n}$. The treatment of diagonals  follows the path of Weaver \cite{Weaver01} in using the symmetric tensor product $\mc{H}\otimes_s\mc{H}$. For example, $D_{1,2} = (\mc{H}_1\otimes_s\mc{H}_2)\otimes \mc{H}_3\otimes\cdots\otimes\mc{H}_n$. We introduce notation to facilitate computations with these diagonals. 

\begin{defn}
For a natural number $n$ and $F\subseteq \{1,\ldots,n\}$ let $\Perm(n)$ be the group of permutations of $\{1,\ldots,n\}$ and $\Perm_n(F)$ be the members of $\Perm(n)$ that fix all elements not in $F$. 
\end{defn}

Let $(a_i)_I$ be an \ts{onb} of $\mc{H}$, then the $a_{\alpha(1)}\otimes\cdots\otimes a_{\alpha(n)}$, where $\alpha\in I^n$ is an $n$-tuple of members of $I$, is an \ts{onb} of $\mc{H}^{\otimes n}$. Thus each $v\in\mc{H}^{\otimes n}$ has a unique representation 
\[ v=\sum_{\alpha\in I^n} \lambda_\alpha\, a_{\alpha_1}\otimes\cdots\otimes a_{\alpha_n} \]
Note that if $\sigma\in$ Perm$(n)$ and $\alpha\in I^n$, the composite $\alpha\sigma=(\alpha_{\sigma(1)},\ldots,\alpha_{\sigma(n)})$ is in $I^n$. 

\begin{defn}
For $(a_i)_I$ an \ts{onb} of $\mc{H}$ and $F\subseteq \{1,\ldots,n\}$ let 
\[D_F = \langle\, \sum_{I^n} \lambda_\alpha a_{\alpha(1)}\otimes\cdots\otimes a_{\alpha(n)} : \lambda_\alpha=\lambda_{\alpha\sigma} \mbox{ for all $\sigma\in $ {\em Perm}$_n(F)$}\rangle\]
\end{defn}

$D_F$ is isomorphic in a natural way to the closed subspace of $\mc{H}^{\otimes n}$ formed by replacing the tensor product of the factors belonging to $F$ with their symmetric tensor product. So the definition is independent of the choice of \ts{onb}. We often use $D_{i,j}$ for $D_{\{i,j\}}$ and so forth. Trivially $D_{i,j}=D_{j,i}$ and $D_{i,i} = \mc{H}^{\otimes n}$. Also, since the group $\Perm_n(\{i,j,k\})$ is generated by the transposition of $i,j$ and the transposition of $j,k$, it follows that $D_{i,j}\cap D_{j,k} = D_{i,j,k}$. 

\begin{prop}
If $j\neq i,k$, then $D_{i,k}=\exists_{\mc{H}_j}(D_{i,j}\cap D_{j,k})$. 
\end{prop}

\begin{proof}
For convenience assume $n=4$. We will show that $D_{1,2}=\exists_{\mc{H}_3}(D_{1,3}\cap D_{3,2})$. That $D_{1,1}=\exists_{\mc{H}_2}(D_{1,2}\cap D_{2,1})$ is similar and other cases are symmetric.
Note that $D_{1,2} =\langle a_i\otimes a_j + a_j\otimes a_i:i,j\in I\rangle \otimes \mc{H}_3\otimes \mc{H}_4$. So $D_{1,2}$ belongs to $\mc{H}_3\otimes \mc{C}(\mc{H}_1\otimes\mc{H}_2\otimes\mc{H}_4)$ and lies above $D_{1,2,3}$. Thus $\exists_{\mc{H}_3}(D_{1,3}\cap D_{3,2})\leq D_{1,2}$. Let $p,q,r,s\in I$ and set 
\[ v = (a_p\otimes a_q\otimes a_q + a_q\otimes a_p\otimes a_q + a_q\otimes a_q\otimes a_p)\otimes a_s. \]
Then $v\in D_{1,2,3}$. To compute $(D_{1,2,3})^{\mc{H}_3}$ express $v=\sum_Ia_i\otimes v_i$ where $a_i$ appears in the third tensor factor. If $p\neq q$, then $v_q = (a_p\otimes a_q + a_q\otimes a_p)\otimes a_s$ and if $p=q$ then $v_q=3a_p\otimes a_p\otimes a_p\otimes a_s$. In either case, the generator $(a_p\otimes a_q+a_q\otimes a_p)\otimes a_r\otimes a_s$ of $D_{1,2}$ belongs to $\mc{H}_3\otimes (D_{1,2,3})^{\mc{H}_3} = \exists_{\mc{H}_3}(D_{1,3}\cap D_{3,2})$. 
\end{proof}

\begin{prop}
$\exists_{\mc{H}_1}(D_{1,2}\cap S)\cap\exists_{\mc{H}_1}(D_{1,2}\cap S^\perp)= 0$ need not hold.
\end{prop}

\begin{proof}
We give an example with $n=2$, but it can be modified for any $n\geq 2$. Let $(a_i)_I$ be an \ts{onb} of $\mc{H}$ and $p,q,r$ be distinct elements of $I$. Set $v=a_p\otimes a_q+a_q\otimes a_p$ and $w=a_p\otimes a_r+a_r\otimes a_p$. Note that $v,w$ are orthogonal and both belong to $D_{1,2}$. Set $S=\langle v\rangle$. Then $S^{\mc{H}_1}=\langle a_p,a_q\rangle$. So $\exists_{\mc{H}_1}(D_{1,2}\cap S) = \exists_{\mc{H}_1}S = \mc{H}_1\otimes \langle a_p,a_q\rangle$. 
Since $w\in D_{1,2}\cap S^\perp$ we have $a_p\in (D_{1,2}\cap S^\perp)^{\mc{H}_1}$. It follows that $\mc{H}_1\otimes \langle a_p\rangle$ is contained in both $\exists_{\mc{H}_1}(D_{1,2}\cap S)$ and $\exists_{\mc{H}_1}(D_{1,2}\cap S^\perp)$. 
\end{proof}

Remark~\ref{rem:substitution} discussed substitution in relation to axiomatics of quantum cylindric algebras. It would be of interest to explore the matter further with an eye towards a version of quantum predicate calculus closed tied to its origins in subspace lattices. 

\begin{remark} {\em 
We indicated that diagonal-free versions of quantum set algebras can be obtained in the infinite-dimensional setting using the complete tensor product $\bigotimes_I\mc{H}_i$. Our discussion of diagonals carries over to a complete tensor power $\mc{H}^{\otimes I}$ without difficulty. If one wanted, the incomplete tensor product could also be used with no essential difficulties, but the complete tensor product seems more natural here.
}
\end{remark}

\begin{remark} {\em 
The path taken to realize diagonal-free quantum cylindric algebras can be extended to the von Neumann algebra setting. Suppose $\mc{M}_1,\ldots\mc{M}_n$ are concrete von Neumann algebras. Then we can define the quantifier $\exists_{\mc{M}_1}$ on the projections $P(\mc{M}_1\otimes\cdots\otimes\mc{M}_n)$ via the complete subalgebra $1\otimes P(\mc{M}_2\otimes\cdots\otimes\mc{M}_n)$ and so forth. We won't develop this further here. 
}
\end{remark}

\section{Kripke frames}

In their work on canonical extensions, J\'{o}nsson and Tarski \cite{JT51a} developed a general method to associate to any \ts{ba} with additional operators $\mc{A}$ a relational structure $\mc{A}_+$, now commonly called a Kripke frame, and to any relational structure $\mc{X}$, a complete atomic \ts{ba} with additional operators $\mc{X}^+$, such that $\mc{A}\leq (\mc{A}_+)^+$. They used this to show that each cylindric algebra can be embedded into a complete atomic cylindric algebra. The frame $\mc{A}_+$ is constructed from $\mc{A}$ via ultrafilters of the \ts{ba} $\mc{A}$ in a way that will not be directly applicable here. Passage from $\mc{X}$ to $\mc{X}^+$ is given as follows. 

\begin{defn}
A relational structure $\mc{X}=(X,(R_i)_I)$ is a set $X$ with a family $(R_i)_I$ of $n_i+1$-ary relations on $X$. Its complex algebra $\mc{X}^+$ is the powerset $P(X)$ with $n_i$-ary operations given by relational image $f_i(A_1,\ldots,A_{n_i})=R_i[A_1,\ldots,A_{n_i}]$ where 
\[R_i[A_1,\ldots,A_{n_i}] = \{x:R_i(a_1,\ldots,a_{n_i},x)\mbox{ for some }a_1\in A_1,\ldots,a_{n_i}\in A_{n_i}\}.\]
\end{defn}

Conditions for $\mc{X}^+$ to be a monadic algebra are in \cite{JT51a} and for $\mc{X}^+$ to be a cylindric algebra in \cite[Thm.~2.7.40]{HMT85a}. An abbreviated account is given below. 

\begin{prop}
For a binary relational structure $\mc{X}=(X,R)$, the complex algebra $\mc{X}^+$ is a monadic algebra iff $R$ is an equivalence relation on $X$. For a relational structure $\mc{X}=(X,(R_i)_I)$, the complex algebra $\mc{X}^+$ is a diagonal-free cylindric algebra iff the $R_i$ are pairwise commuting equivalence relations. 
\end{prop}

We make first steps to develop the connection between monadic and cylindric \ts{ol}'s and frames. ``Under the hood'' our approach departs from that of J\'{o}nsson and Tarski in that it is based on MacNeille completions rather than canonical extensions. 

\begin{defn}
An orthoframe $\mc{X}=(X,\perp)$ is a set with an irreflexive, symmetric binary relation called an orthogonality relation. For $A\subseteq X$ set 
$$A^\perp = \{x\in X : a\perp x \mbox{ for all }a\in A\}$$ 
and let $\mc{L}(X,\perp) = \{A:A=A^{\perp\perp}\}$ be all biorthogonally closed subsets of $X$. 
\end{defn}

It is well-known \cite{Kalmbach83} that $\mc{L}(X,\perp)$ is a complete \ts{ol} with meets given by intersections, joins by the biorthogonal of the union, and with $A^\perp$ the orthocomplement of $A$. If $L$ is an \ts{ol}, then $(L^*,\perp)$ is an orthoframe where $L^*=L\setminus \{0\}$ and $x\perp y$ iff $x\leq y^\perp$. In this case $\mc{L}(L^*,\perp)$ is the MacNeille completion of $L$, and in particular is a complete \ts{ol} containing $L$ as a sub-\ts{ol}. Unfortunately, it is not possible to give first-order properties of $(X,\perp)$ equivalent to $\mc{L}(X,\perp)$ being an \ts{oml} \cite{Goldblatt84}. 

\begin{defn}\label{defn:monadicorthoframe}
A relational structure $(X,\perp,R)$ is a monadic orthoframe if $\perp$ and $R$ are binary relations on $X$ that satisfy
\begin{itemize}[leftmargin=.4in]
\item[{\em \small (M$_1$)}] $\perp$ is an orthogonality relation,
\item[{\em \small (M$_2$)}] $R$ is reflexive and transitive,
\item[{\em \small (M$_3$)}] for each $x\in X$, the set $R[\{x\}]^\perp$ is closed under $R$.
\end{itemize}
Here {\em \small (M$_3$)} means $R[R[\{x\}]^\perp] \subseteq R[\{x\}]^\perp$ for each $x\in X$. Since {\em \small (M$_2$)} gives reflexivity, set inclusion here can be replaced with equality. 
\end{defn}

For any set $X$, the relation $x\perp y \Leftrightarrow x\neq y$ is an orthogonality relation that we denote $\neq$. The \ts{ol} $\mc{L}(X,\neq)$ is the powerset of $X$. If we consider orthoframes restricted to having this classical orthogonality relation, we obtain exactly the frames corresponding to classical monadic algebras as the following result shows. 

\begin{prop}
For any set $X$, the frame $(X,\neq,R)$ is a monadic orthoframe iff $R$ is an equivalence relation on $X$. 
\end{prop}

\begin{proof}
Suppose that $R$ is an equivalence relation on $X$. Clearly $\neq$ is an orthogonality relation and $R$ is reflexive and transitive. For $x\in X$ we have $R[\{x\}]$ is the equivalence class of $x$, and since $\perp$ is $\neq$, we have $R[\{x\}]^\perp$ is the set-theoretic complement of this equivalence class, which is indeed closed under $R$. So $(X,\neq,R)$ is a monadic orthoframe. Conversely, suppose $(X,\neq,R)$ is a monadic orthoframe. By definition $R$ is reflexive and transitive. Suppose $x\,R\,y$. If it is not the case that $y\,R\, x$, then $y$ belongs to the set-theoretic complement of $R[\{x\}]$ which in the current setting is given by $R[\{x\}]^\perp$. Since $R[\{x\}]^\perp$ is closed under $R$ and $x\not\in R[\{x\}]^\perp$ (since $R$ is reflexive and $\perp$ is irreflexive), it is not the case that $y\, R\, x$. So $R$ is symmetric, hence an equivalence relation. 
\end{proof}

Note that Definition~\ref{defn:monadicorthoframe} is within the first-order language of the relational structure since quantification in {\small (M$_3$)} is over individuals in $X$ and not subsets of $X$. However, the conditions are sufficient to give the version of {\small (M$_3$)} when quantified over all subsets and not just singletons. This is the content of the first statement of the following result. 

\begin{lemma} \label{lem:R}
Let $(X,\perp,R)$ be a monadic orthoframe. Then for each $A\subseteq X$
\begin{itemize}
\item[{\em\small (1)}] $R[A]^\perp$ is closed under $R$,
\item[{\em\small (2)}] $R[A]^{\perp\perp}$ is closed under $R$,
\item[{\em\small (3)}] $R[A^{\perp\perp}]\subseteq R[A]^{\perp\perp}$. 
\end{itemize}
Hence, as $R$ is reflexive, $R[R[A]^\perp]=R[A]^\perp$ and $R[R[A]^{\perp\perp}]=R[A]^{\perp\perp}$. 
\end{lemma}

\begin{proof}
(1) Since $R[A]=\bigcup\{R[\{x\}]:x\in A\}$ then $R[A]^\perp = \bigcap\{R[\{x\}]^\perp:x\in A\}$. So if $y\in R[A]^\perp$ and $y\,R\,z$, then $y\in R[\{x\}]^\perp$ for all $x\in A$, and since each $R[\{x\}]^\perp$ is closed under $R$, it follows that $z\in R[A]^\perp$. This gives the first statement and as $R$ is reflexive, gives $R[R[A]^\perp]=R[A]^\perp$. (2) It is convenient to write $A^\perp$ as $OA$ and $R[A]$ as $RA$. Then the first statement gives $RORA=ORA$ for all $A\subseteq X$. Then 
\[ ROORA = RORORA =  ORORA = OORA \]
giving $R[R[A]^{\perp\perp}] = R[A]^{\perp\perp}$ and hence that $R[A]^{\perp\perp}$ is closed under $R$. (3)~Since $R$ is reflexive, $A\subseteq R[A]$, and $A\subseteq A^{\perp\perp}$ since biorthogonal is a closure operator. So 
$R[A^{\perp\perp}]\subseteq R[R[A]^{\perp\perp}] = R[A]^{\perp\perp}$ where the final equality uses the second statement. 
\end{proof}

\begin{defn}
For $\mc{X}=(X,\perp,R)$ an orthoframe with additional binary relation $R$, let $\exists_R$ be the unary operation on $\mc{L}(X,\perp)$ where $\exists_R A = R[A]^{\perp\perp}$. Let $\mc{L}(\mc{X}) = (\mc{L}(X,\perp),\exists_R)$. 
\end{defn}

\begin{prop}
If $\mc{X}=(X,\perp,R)$ is a monadic orthoframe, $\mc{L}(\mc{X})$ is a monadic \ts{ol}. 
\end{prop}

\begin{proof}
Equations {\small (Q$_1$)} and {\small (Q$_2$)} are trivial. For {\small (Q$_3$)}, note that $\exists_R$ is order preserving, so $\exists_R A \vee \exists_R B \leq \exists_R (A\vee B)$. For the other inequality, using Lemma~\ref{lem:R}.3 and that the biorthogonal is a closure operator
\[ \exists_R(A\vee B) = R[(A\cup B)^{\perp\perp}]^{\perp\perp} \subseteq R[A\cup B]^{\perp\perp\perp\perp} = R[A\cup B]^{\perp\perp}.\]
Then since $R[A\cup B]=R[A]\cup R[B]$, using basic properties of the biorthogonal 
\[R[A\cup B]^{\perp\perp} = (R[A]\cup R[B])^{\perp\perp} \subseteq (R[A]^{\perp\perp}\cup R[B]^{\perp\perp})^{\perp\perp}.\]
This last expression is $\exists_RA\vee\exists_RB$, providing {\small (Q$_3$)}. Using Lemma~\ref{lem:R}.2 and that biorthogonal is a closure operator gives 
\[ \exists_R\exists_R A = R[R[A]^{\perp\perp}]^{\perp\perp} = R[A]^{\perp\perp\perp\perp} = R[A]^{\perp\perp} = \exists_RA,\]
yielding {\small (Q$_4$)}. Finally, using Lemma~\ref{lem:R}.1 and the fact that $A^{\perp\perp\perp}=A^\perp$ 
\[\exists_R(\exists_RA)^\perp = R[R[A]^{\perp\perp\perp}]^{\perp\perp}=R[R[A]^\perp]^{\perp\perp} = R[A]^{\perp\perp\perp} = R[A]^\perp. \]
This gives the final condition {\small (Q$_5$)}.
\end{proof}

\begin{prop}\label{prop:submonadic}
Every monadic \ts{ol} can be embedded into $\mc{L}(\mc{X})$ for some monadic orthoframe $\mc{X}$ and every complete monadic \ts{ol} is isomorphic to some such $\mc{L}(\mc{X})$. 
\end{prop}

\begin{proof}
Given a monadic \ts{ol} $(L,\exists)$ set $\mc{X}=(L^*,\perp,R)$ where $L^*=L\setminus\{0\}$, $x\perp y$ iff $x\leq y^\perp$, and $x\,R\,y$ iff $y\leq \exists x$. Since $\exists$ is increasing and idempotent, $R$ is reflexive and transitive. Note that $R[\{x\}]=\{y:y\leq\exists x\}$, so $R[\{x\}]^\perp=\{y:y\leq(\exists x)^\perp\}$. So if $y\in R[\{x\}]^\perp$ and $y\,R\,z$, then $y\leq(\exists x)^\perp$ and $z\leq\exists y$, so $z\leq \exists(\exists x)^\perp$. Then by {\small (Q$_5$)} $z\leq(\exists x)^\perp$, so $z\in R[\{x\}]^\perp$. Thus $R[\{x\}]^\perp$ is closed under $R$. 

It is well-known that $L$ embeds into $\mc{L}(X,\perp)$ via $\alpha(a)={\downarrow a}=\{x:x\leq a\}$ and $\alpha$ is an \ts{ol}-isomorphism if $L$ is complete. It remains to show that $\alpha(\exists a)=\exists_R (\alpha a)$. But 
$\exists_R(\alpha a) = R[\,{\downarrow} a\,]^{\perp\perp} = ({\downarrow} \exists a)^{\perp\perp} = {\downarrow}\exists a = \alpha(\exists a)$ since each principle downset is biorthogonally closed. 
\end{proof}

\begin{defn}
For a set $I$, a relational structure $\mc{X}=(X,\perp,(R_i)_I,(D_{i,j})_{I\times J})$ is an $I$-dimensional weak cylindric orthoframe if for each $i,j,k\in I$
\begin{itemize}[leftmargin=.45in]
\item[{\em \small (W$_1$)}] $(X,\perp,R_i)$ is a monadic orthoframe,
\item[{\em \small (W$_2$)}] $R_i$ commutes with $R_j$,
\item[{\em \small (W$_3$)}] $D_{i,j}=D_{j,i} = D_{i,j}^{\perp\perp}$ and $D_{i,i}=X$,
\item[{\em \small (W$_4$)}] if $j\neq i,k$ then $R_j[D_{i,j}\cap D_{j,k}]=D_{i,k}$. 
\end{itemize}\end{defn}

The definition of $\mc{L}(\mc{X})$ for a weak cylindric orthoframe $\mc{X}$ is the obvious one. 

\begin{prop}
If $\mc{X}=(X,\perp,(R_i)_I,(D_{i,j})_{I\times J})$ is a weak cylindric orthoframe, then $\mc{L}(\mc{X})$ is a weak cylindric \ts{ol} and each weak cylindric \ts{ol} is a subalgebra of some $\mc{L}(\mc{X})$. 
\end{prop}

\begin{proof}
To show $\mc{L}(\mc{X})$ is a weak cylindric \ts{ol} we show that the quantifiers $\exists_{R_i}$ commute. The other properties are fairly obvious. Using Lemma~\ref{lem:R}.3 and that $R_i,R_j$ commute, 
\[ R_i[R_j[A]^{\perp\perp}]^{\perp\perp}\subseteq R_i[R_j[A]]^{\perp\perp}= R_j[R_i[A]]^{\perp\perp}\subseteq R_j[R_i[A]^{\perp\perp}]^{\perp\perp}.\]
So $\exists_{R_i}\exists_{R_j}A\subseteq\exists_{R_j}\exists_{R_i}A$, and by symmetry, we have equality. 
Conversely, given a weak cylindric \ts{ol} $(L,(\exists_i)_I,(d_{i,j})_{I\times J})$, let $\mc{X}=(L^*,\perp,(R_i)_I,(D_{i,j})_{I\times J})$ where $L^*$, $\perp$, and each $R_i$ is as in the proof of Proposition~\ref{prop:submonadic} and $D_{i,j}={\downarrow }\,d_{i,j}$. We show $R_i,R_j$ commute, the other properties are obvious. If $x\, R_i\, y$ and $y\, R_j z$, then $y\leq\exists_i\, x$ and $z\leq \exists_j\, y$. Then $z\leq \exists_j\exists_i\, x$. Since the quantifiers commute, $z\leq \exists_i\exists_j\, x$, so for $w=\exists_j\, x$ we have $x\, R_j\, w$ and $w\, R_i\, z$. Thus $R_i\circ R_j\subseteq R_j\circ R_i$, and by symmetry they are equal. 
\end{proof}
\vspace{-2ex}

\section{Conclusion}

We introduced monadic and cylindric \ts{ol}'s and developed their basic theory. The connection between these structures and von Neumann algebras was illustrated through a series of examples. A version of quantum cylindric set algebras based in tensor powers of a Hilbert space was developed and related to algebraic quantum predicate calculus. Issues with substitution in this context were observed. A form of Kripke frames for these structures was developed. 

Items related to the discussion deserve additional attention. A path was suggested to determine if each monadic \ts{ol} is isomorphic to a functional one. The relationship between commuting squares of subfactors and quantum cylindric algebras should be settled. Axiomatic issues were found in quantum cylindric algebras related to substitution operations and quantum predicate calculus. Kripke frames developed for monadic \ts{ol}s could be developed from the perspective of canonical extensions rather than MacNeille completions as was done here. 

The broad aim is to use the structures introduced here as vehicles to view modern developments in operator algebras from the perspective of order theory, geometry, and logic. 


\begin{center} {\large \ts{References} } \end{center}
\vspace{1ex}

\bibliographystyle{plain}
\bibliography{Q-monadic}{}

\begin{thebibliography}{10}

\bibitem{AP10}
Claire Anantharaman and Sorin Popa.
\newblock An introduction to type ii$_1$ factors, February 2010.

\bibitem{BH02}
Guram Bezhanishvili and John Harding.
\newblock Functional monadic {H}eyting algebras.
\newblock {\em Algebra Universalis}, 48(1):1--10, 2002.

\bibitem{BvN37}
Garrett Birkhoff and John von Neumann.
\newblock The logic of quantum mechanics.
\newblock {\em Ann. of Math. (2)}, 37(4):823--843, 1936.

\bibitem{Blackadar06}
B.~Blackadar.
\newblock {\em Operator algebras}, volume 122 of {\em Encyclopaedia of
  Mathematical Sciences}.
\newblock Springer-Verlag, Berlin, 2006.
\newblock Theory of $C^*$-algebras and von Neumann algebras, Operator Algebras
  and Non-commutative Geometry, III.

\bibitem{brv01}
Patrick Blackburn, Maarten de~Rijke, and Yde Venema.
\newblock {\em Modal logic}, volume~53 of {\em Cambridge Tracts in Theoretical
  Computer Science}.
\newblock Cambridge University Press, Cambridge, 2001.

\bibitem{BJ72}
T.~S. Blyth and M.~F. Janowitz.
\newblock {\em Residuation theory}.
\newblock International Series of Monographs in Pure and Applied Mathematics,
  Vol. 102. Pergamon Press, Oxford-New York-Toronto, Ont., 1972.

\bibitem{BH00}
Gunter Bruns and John Harding.
\newblock Algebraic aspects of orthomodular lattices.
\newblock In {\em Current research in operational quantum logic}, volume 111 of
  {\em Fund. Theories Phys.}, pages 37--65. Kluwer Acad. Publ., Dordrecht,
  2000.

\bibitem{BS81}
Stanley Burris and H.~P. Sankappanavar.
\newblock {\em A course in universal algebra}, volume~78 of {\em Graduate Texts
  in Mathematics}.
\newblock Springer-Verlag, New York-Berlin, 1981.

\bibitem{CH02}
G.~Cattaneo and J.~Hamhalter.
\newblock De {M}organ property for effect algebras of von {N}eumann algebras.
\newblock {\em Lett. Math. Phys.}, 59(3):243--252, 2002.

\bibitem{Cirulis94}
J\={a}nis C\={\i}rulis.
\newblock Abstract algebras of finitary relations: several non-traditional
  axiomatizations.
\newblock In {\em Mathematics}, volume 595 of {\em Latv. Univ. Zin\={a}t.
  Raksti}, pages 23--48. Latv. Univ., Riga, 1994.

\bibitem{DGG04}
M.~Dalla~Chiara, R.~Giuntini, and R.~Greechie.
\newblock {\em Reasoning in quantum theory}, volume~22 of {\em Trends in
  Logic---Studia Logica Library}.
\newblock Kluwer Academic Publishers, Dordrecht, 2004.
\newblock Sharp and unsharp quantum logics.

\bibitem{Dvurecenskij93}
Anatolij Dvure\v{c}enskij.
\newblock {\em Gleason's theorem and its applications}, volume~60 of {\em
  Mathematics and its Applications (East European Series)}.
\newblock Kluwer Academic Publishers Group, Dordrecht; Ister Science Press,
  Bratislava, 1993.

\bibitem{Foulis60}
David~J. Foulis.
\newblock Baer {$^{\ast} $}-semigroups.
\newblock {\em Proc. Amer. Math. Soc.}, 11:648--654, 1960.

\bibitem{Galler57}
Bernard~A. Galler.
\newblock Cylindric and polyadic algebras.
\newblock {\em Proc. Amer. Math. Soc.}, 8:176--183, 1957.

\bibitem{Goldblatt84}
Robert Goldblatt.
\newblock Orthomodularity is not elementary.
\newblock {\em J. Symbolic Logic}, 49(2):401--404, 1984.

\bibitem{GPJ89}
Frederick~M. Goodman, Pierre de~la Harpe, and Vaughan F.~R. Jones.
\newblock {\em Coxeter graphs and towers of algebras}, volume~14 of {\em
  Mathematical Sciences Research Institute Publications}.
\newblock Springer-Verlag, New York, 1989.

\bibitem{Halmos62}
Paul~R. Halmos.
\newblock {\em Algebraic logic}.
\newblock Chelsea Publishing Co., New York, 1962.

\bibitem{Hamhalter03}
Jan Hamhalter.
\newblock {\em Quantum measure theory}, volume 134 of {\em Fundamental Theories
  of Physics}.
\newblock Kluwer Academic Publishers Group, Dordrecht, 2003.

\bibitem{Harding91}
John Harding.
\newblock Orthomodular lattices whose {M}ac{N}eille completions are not
  orthomodular.
\newblock {\em Order}, 8(1):93--103, 1991.

\bibitem{Harding98}
John Harding.
\newblock Canonical completions of lattices and ortholattices.
\newblock volume~15, pages 85--96. 1998.
\newblock Quantum structures, II (Liptovsk\'{y} J\'{a}n, 1998).

\bibitem{Harding17}
John Harding.
\newblock Dynamics in the decompositions approach to quantum mechanics.
\newblock {\em Internat. J. Theoret. Phys.}, 56(12):3971--3990, 2017.

\bibitem{HMT85a}
Leon Henkin, J.~Donald Monk, and Alfred Tarski.
\newblock {\em Cylindric algebras. {P}art {I}}, volume~64 of {\em Studies in
  Logic and the Foundations of Mathematics}.
\newblock North-Holland Publishing Co., Amsterdam, 1985.

\bibitem{HMT85b}
Leon Henkin, J.~Donald Monk, and Alfred Tarski.
\newblock {\em Cylindric algebras. {P}art {II}}, volume 115 of {\em Studies in
  Logic and the Foundations of Mathematics}.
\newblock North-Holland Publishing Co., Amsterdam, 1985.

\bibitem{JS97}
V.~Jones and V.~S. Sunder.
\newblock {\em Introduction to subfactors}, volume 234 of {\em London
  Mathematical Society Lecture Note Series}.
\newblock Cambridge University Press, Cambridge, 1997.

\bibitem{JT51a}
Bjarni J\'{o}nsson and Alfred Tarski.
\newblock Boolean algebras with operators. {I}.
\newblock {\em Amer. J. Math.}, 73:891--939, 1951.

\bibitem{Kadison04}
Richard~V. Kadison.
\newblock Non-commutative conditional expectations and their applications.
\newblock In {\em Operator algebras, quantization, and noncommutative
  geometry}, volume 365 of {\em Contemp. Math.}, pages 143--179. Amer. Math.
  Soc., Providence, RI, 2004.

\bibitem{KR97a}
Richard~V. Kadison and John~R. Ringrose.
\newblock {\em Fundamentals of the theory of operator algebras. {V}ol. {I}},
  volume~15 of {\em Graduate Studies in Mathematics}.
\newblock American Mathematical Society, Providence, RI, 1997.

\bibitem{KR97b}
Richard~V. Kadison and John~R. Ringrose.
\newblock {\em Fundamentals of the theory of operator algebras. {V}ol. {II}},
  volume~16 of {\em Graduate Studies in Mathematics}.
\newblock American Mathematical Society, Providence, RI, 1997.

\bibitem{Kalmbach83}
Gudrun Kalmbach.
\newblock {\em Orthomodular lattices}, volume~18 of {\em London Mathematical
  Society Monographs}.
\newblock Academic Press, Inc. [Harcourt Brace Jovanovich, Publishers], London,
  1983.

\bibitem{MV57}
A.~Monteiro and O~Varsavsky.
\newblock Algebras de heyting monadicas.
\newblock {\em Actas de las X Jornadas de la Union Matematica Argentina}, pages
  52--62, 1957.

\bibitem{Naaijkens17}
Peter. Naaijkens.
\newblock Subfactors and quantum information theory, 2017.

\bibitem{Nakagami70}
Yoshiomi Nakagami.
\newblock Infinite tensor products of von {N}eumann algebras. {I}.
\newblock {\em Kodai Math. Sem. Rep.}, 22:341--354, 1970.

\bibitem{Parthasarathy92}
K.~R. Parthasarathy.
\newblock {\em An introduction to quantum stochastic calculus}.
\newblock Modern Birkh\"{a}user Classics. Birkh\"{a}user/Springer Basel AG,
  Basel, 1992.
\newblock [2012 reprint of the 1992 original] [MR1164866].

\bibitem{PP86}
Mihai Pimsner and Sorin Popa.
\newblock Entropy and index for subfactors.
\newblock {\em Ann. Sci. \'{E}cole Norm. Sup. (4)}, 19(1):57--106, 1986.

\bibitem{Popa83}
Sorin Popa.
\newblock Orthogonal pairs of {$\ast $}-subalgebras in finite von {N}eumann
  algebras.
\newblock {\em J. Operator Theory}, 9(2):253--268, 1983.

\bibitem{Speicher16}
R.~Speicher.
\newblock Lecture notes on von neumann algebras, subfactors, knots and braids,
  and planar algebras, February 2016.

\bibitem{TW01}
T.~Thiemann and O.~Winkler.
\newblock Gauge field theory coherent states ({GCS}). {IV}. {I}nfinite tensor
  product and thermodynamical limit.
\newblock {\em Classical Quantum Gravity}, 18(23):4997--5053, 2001.

\bibitem{Uhlhorn63}
U.~Uhlhorn.
\newblock Representation of symmetry transformations in quantum mechanics.
\newblock {\em Ark. Fys.}, 23:307--340, 1963.

\bibitem{Varadarajan68}
V.~S. Varadarajan.
\newblock {\em Geometry of quantum theory. {V}ol. {I}}.
\newblock The University Series in Higher Mathematics. D. Van Nostrand Co.,
  Inc., Princeton, N.J.-Toronto, Ont.-London, 1968.

\bibitem{vN39}
J.~von Neumann.
\newblock On infinite direct products.
\newblock {\em Compositio Math.}, 6:1--77, 1939.

\bibitem{WW95}
Yasuo Watatani and Jerzy Wierzbicki.
\newblock Commuting squares and relative entropy for two subfactors.
\newblock {\em J. Funct. Anal.}, 133(2):329--341, 1995.

\bibitem{Weaver01}
Nik Weaver.
\newblock {\em Mathematical quantization}.
\newblock Studies in Advanced Mathematics. Chapman \& Hall/CRC, Boca Raton, FL,
  2001.

\bibitem{Westerbaan18}
Abraham~A. Westerbaan.
\newblock The category of von neumann algebras, 2018.

\end{thebibliography}

\end{document}